\documentclass[twoside]{article}
\usepackage{amsfonts,latexsym,amsmath,amsthm,amscd,eucal,amssymb}

\catcode`\@=11
\long\def\@makefntext#1{
\protect\noindent \hbox to 3.2pt {\hskip-.9pt  
$^{{\eightrm\@thefnmark}}$\hfil}#1\hfill}               

\def\@makefnmark{\hbox to 0pt{$^{\@thefnmark}$\hss}}    
        
\def\ps@myheadings{\let\@mkboth\@gobbletwo
  \def\@oddhead{{\slshape\rightmark}\hfil{\footnotesize\thepage}}
  \def\@oddfoot{}
  \def\@evenhead{{\footnotesize\thepage}\hfil\slshape\leftmark}
  \def\@evenfoot{}
  \def\sectionmark##1{}\def\subsectionmark##1{}
}

\oddsidemargin=\evensidemargin
\numberwithin{equation}{section}

\renewcommand\section{\@startsection {section}{1}{\z@}%
                                   {-3.5ex \@plus -1ex \@minus -.2ex}%
                                   {2.3ex \@plus.2ex}%
                                   {\tenbf\large\bfseries}}
\renewcommand\subsection{\@startsection{subsection}{2}{\z@}%
                                     {-3.25ex\@plus -1ex \@minus -.2ex}%
                                     {1.5ex \@plus .2ex}%
                                     {\normalfont\bfseries}}

\newcommand{\textlineskip}{\baselineskip=13pt}
\newcommand{\smalllineskip}{\baselineskip=10pt}

\newcommand{\copyrightheading}[1]
        {\vspace*{-2.5cm}\smalllineskip{\flushleft
        {\footnotesize Dipartimento di Matematica, Universit\`a di Trento}\\
        {\footnotesize Preprint UTM - #1}\\
        {\tiny 
          \number\day/\number\month/\number\year
        }%
        }}
\newcommand{\timenow}{%
  \@tempcnta=\time \divide\@tempcnta by 60 \number\@tempcnta:\multiply
  \@tempcnta by 60 \@tempcntb=\time \advance\@tempcntb by -\@tempcnta
  \ifnum\@tempcntb <10 0\number\@tempcntb\else\number\@tempcntb\fi}

\def\abstracts#1#2#3{{
        \centering{\begin{minipage}{4.5in}\footnotesize\baselineskip=10pt
        \parindent=0pt #1\par 
        \parindent=15pt \footnotesize\baselineskip=10pt
        {\footnotesize\it Keywords}\/: #2\par
        \parindent=15pt \footnotesize\baselineskip=10pt
        {\footnotesize\it 1991 MSC}\/: #3
        \end{minipage}}\par}} 

\renewenvironment{thebibliography}[1]
        {\frenchspacing
         \ninerm\baselineskip=11pt
         \begin{list}{\arabic{enumi}.}
        {\usecounter{enumi}\setlength{\parsep}{0pt}     
         \setlength{\leftmargin 17pt}{\rightmargin 0pt}   
         \setlength{\itemsep}{0pt} \settowidth
        {\labelwidth}{#1.}\sloppy}}{\end{list}}

\def\pmb#1{\setbox0=\hbox{#1}
        \kern-.025em\copy0\kern-\wd0
        \kern.05em\copy0\kern-\wd0
        \kern-.025em\raise.0433em\box0}

\def\fnt#1#2{\footnotetext{\kern-.3em
        {$^{\mbox{\scriptsize #1}}$}{#2}}}

\def\runninghead#1#2{\pagestyle{myheadings}
\markboth{{\protect\footnotesize\it{\quad #1}}}
{{\protect\footnotesize\it{#2\quad}}}}
\font\tenbf=cmbx10

\font\ninerm=cmr9

\font\eightrm=cmr8

\newtheorem{theorem}{Theorem}[section]

\newtheorem{proposition}[theorem]{Proposition}
\newtheorem{definition}[theorem]{Definition}
\newtheorem{corollary}[theorem]{Corollary}
\newtheorem{lemma}[theorem]{Lemma}

\newtheorem{hypothesis}{Hypothesis}[section]

\newtheorem{remark}{Remark}[section]

\def\Om{{\Omega}}

\def\R{{\mathbb R}}

\def\ep{\varepsilon}
\def\P{\mathbb{P}}

\def\EE{{\mathbb E}}

\def\lb{{\lambda}}
\def\N{{\mathbb{N}}}

\begin{document}
\setlength{\textheight}{7.7truein}

\runninghead{S. Bonaccorsi, E. Mastrogiacomo}{Analysis of the
  stochastic FitzHugh-Nagumo system}

\normalsize\textlineskip
\thispagestyle{empty}
\setcounter{page}{1}

\copyrightheading{719}

\vspace*{.8truein}

\centerline{\bf Analysis of the stochastic FitzHugh-Nagumo system}
\baselineskip=13pt
\vspace*{0.3truein}

\centerline{\footnotesize Stefano BONACCORSI\footnotemark[1] \quad
  Elisa MASTROGIACOMO\footnotemark[2]}
\baselineskip=12pt
\centerline{\footnotesize\it Dipartimento di Matematica, Universit\`a
  di Trento,}
\baselineskip=10pt 
\centerline{\footnotesize\it via Sommarive 14, 38050 Povo (Trento),
  Italia}
\baselineskip=10pt 
\centerline{\footnotemark[1] \footnotesize\tt
  stefano.bonaccorsi@unitn.it}
\centerline{\footnotemark[2] \footnotesize\tt
  mastrogiacomo@science.unitn.it}
\vspace*{0.2truein}
\abstracts
{In this paper we study a system of stochastic differential equations
  with dissipative nonlinearity which arise in certain neurobiology
  models. Besides proving existence, uniqueness and continuous
  dependence on the initial datum, we shall be mainly concerned with
  the asymptotic behaviour of the solution. We prove the existence of
  an invariant ergodic measure $\nu$ associated with the transition
  semigroup $P_t$; further, we identify its infinitesimal generator in
  the space $L^2(H;\nu)$.}
{Stochastic FitzHugh-Nagumo system; Invariant measures; Wiener process;
Transition semigroup; Kolmogorov operator
}
{35K57; 60H15; 37L40}

\vspace*{4pt}
\baselineskip=13pt              
\normalsize                     


\section{Introduction}

Since the fundamental work of Hodgkin and Huxley
\cite{hodgkin/huxley}, several equations were proposed to model the
behavior of the signal propagation in a neural cell. The original
model was formed by a system of four equations, the first one modeling
how the electric impulses propagate along a long tube (the axon) that
we model as a (normalized) segment $(0,1)$, while the remaining were
concerned with various ions concentrations in the cell. Later, a more
analytically tractable model were proposed by FitzHugh \cite{fitzhugh}
and Nagumo \cite{nagumo}. In this paper we focus the interest on a
stochastic version of the FitzHugh-Nagumo model. It consists of two
variables, the first one, $u$, represents the {\em voltage variable}
and the second one, $w$, is the {\em recovery variable}, associated
with the concentration of potassium ions in the axon.  For a thorough
introduction to the biological motivations of this model we refer to
Murray \cite{murray} or Keener and Sneyd \cite{keener/sneyd}.

Let us consider the equation
\begin{equation}\label{eq:mot_ex}
  \begin{aligned}
    \partial_t u(t,\xi)&= \partial_{\xi} \left(c(\xi)\partial_{\xi}
    u(t,\xi) \right) - p(\xi)u(t,\xi) + f(u(t,\xi)) - w(t,\xi)
    \\
    &\phantom{= \partial_{\xi} \left(c(\xi)\partial_{\xi}
    u(t,\xi) \right) -}  + \partial_t
    \beta_1(t,\xi), \qquad t\geq 0,\quad \xi\in
    [0,1],
    \\
    \partial_t w(t,\xi) &= -\alpha w(t,\xi) + \gamma u(t,\xi) +
    \partial_t \beta_2(t,\xi),
    \\
    &\phantom{= \partial_{\xi} \left(c(\xi)\partial_{\xi}
    u(t,\xi) \right) -  + \partial_t
    \beta_1(t,\xi),} \quad t\geq 0, \quad \xi\in
    [0,1],
  \end{aligned}
\end{equation}    
where $u$ represents the electrical potential and $w$ is the recovery
variable; $\alpha$, $\gamma$, $c(\xi)$ and $p(\xi)$ are given
phenomenological coefficients satisfying the conditions stated below;
$\beta_1,\beta_2$ are independent Brownian motions; $f$ is a nonrandom
real-valued function with suitable smoothness properties: in the
reduced FitzHugh-Nagumo system $f$ is a polynomial of odd degree,
precisely $f(u)=-u(u-1)(u-\xi_1)$, where $0 < \xi_1 < 1$ represents
the voltage threshold. Problem (\ref{eq:mot_ex}) shall be endowed
with boundary and initial conditions. The first one are necessary only
for the potential $u(t,\xi)$ and we assume they are of Neumann type:
$\partial_\xi u(t,0) = \partial_\xi u(t,1) = 0$; the initial
condition are given, for simplicity, by continuous functions 
\begin{equation*}
  u(0,\xi) = u_0(\xi), \qquad v(0,\xi) = v_0(\xi)
\end{equation*}
with $u_0, v_0 \in C([0,1])$.

We shall introduce the main assumptions on the coefficients of problem
(\ref{eq:mot_ex}) that will be used without stating in the following.
For this, it is necessary to introduce the operator $A_0$ on the space
$L^2(0,1)$, defined on $D(A_0) = \{u \in H^2(0,1) \mid \partial_\xi
u(\zeta) = 0,\ \zeta=0,1\}$ by
\begin{equation*}
  A_0 u(\xi) = \partial_{\xi} \left(c(\xi)\partial_{\xi}
    u(\xi)\right), \qquad \xi \in [0,1],\ u \in D(A_0).
\end{equation*}

\begin{hypothesis}\label{hp:1}
  The constants $\alpha$ and $\gamma$ are strictly positive real
  numbers; the functions $c(\xi)$ and $p(\xi)$ belong to $C^1([0,1])$,
  $c = \min\limits_{[0,1]} c(\xi) > 0$ and $p = \min\limits_{[0,1]}
  p(\xi) > 0$. Further, for $\xi_1$ from the definition of the
  FitzHugh-Nagumo nonlinearity, it holds
  \begin{equation}\label{eq:1}
    3p-(\xi_1^2-\xi_1+1) \geq 0.
  \end{equation}
  There exists a complete orthonormal basis $\left\{e_k\right\}$ of
  $L^2(0,1)$ made of eigenvectors of $A_0$, such that the $\{e_k\}$
  satisfy a uniform bound in the sup-norm, i.e., for some $M > 0$ it
  holds
  \begin{equation}\label{eq:stima_e_k}
    \left|e_k(\xi)\right|\leq M, \qquad \xi \in [0,1],\ k \in {\mathbb
    N}.
  \end{equation}
  Let $\beta_1$, $\beta_2$ be independent Wiener processes on a
  filtered probability space $(\Om, \mathcal{F}, \mathcal{F}_t,
  \mathbb{P})$ with continuous trajectories on $[0,T]$ for any $T>0$;
  this means that 
  \begin{equation*}
    \beta_i \in C([0,T]; L^2(\Om, L^2(0,1)))
  \end{equation*}
  with $\mathcal{L}(\beta_i (t,\cdot)) \cong
  \mathcal{N}(0,t\sqrt{Q_i})$ for suitable linear operators $Q_i$,
  $i=1,2$ on $L^2(0,1)$.
  
  With no loss of generality we can assume that the operators $Q_i$,
  $i=1,2$ diagonalize on the same basis $\left\{e_k\right\}$.
  Therefore, there exist sequences $\lambda_k^i$, $i=1,2$, $k \in \N$,
  of positive real numbers such that
  \begin{equation*}
      Q_i e_k = \lb_k^i e_k, \qquad i=1,2, \quad k=1, 2, \dots.
  \end{equation*}
  Furthermore, we assume that 
  \begin{equation*}
    \sum_{i=1}^2 \sum_{k=1}^\infty \lb_k^i < \infty;
  \end{equation*}
  hence ${\rm Tr}Q_i < \infty$.
\end{hypothesis}

It is convenient to write \eqref{eq:mot_ex} in an abstract form. To
this end we set $H=L^2(0,1)\times L^2(0,1)$ endowed with the 
inner product 
\begin{equation*}
  \left\langle (u_1,w_1), (u_2,w_2) \right\rangle_H = \gamma \,
  \langle u_1, u_2 \rangle_{L^2} + \langle w_1, w_2 \rangle_{L^2}
\end{equation*}
where $\langle \cdot,\cdot\rangle_{L^2}$ is the usual scalar product
in $L^2(0,1)$ and $\gamma$ is the constant from (\ref{eq:mot_ex}). The
corresponding norm is denoted by $\left|\cdot\right|_H$.
We also introduce the space $V = H^1(0,1)\times L^2(0,1)$ with the norm
\begin{equation*}
  \|x\|^2_V = \gamma \left|x_1\right|^2_{H^1} +
  \left|x_2\right|^2_{L^2}.
\end{equation*}

On the space $H$ we introduce the following operators: 
\begin{equation}
  \begin{aligned}
    &A \, : \, D(A) \subset H \, \to \, H,\quad D(A) = D(A_0) \times
    L^2(0,1)
    \\
    &A \left({u \atop w}\right) = \begin{pmatrix} A_0 u & -w
      \\
      \gamma u & -\alpha w
    \end{pmatrix}
  \end{aligned}
\end{equation}
and
\begin{equation}
  \begin{aligned}
    &F\, : \, D(F) := L^6(0,1)\times\, L^2(0,1)\to\, H \\
    &F \left({u \atop w}\right) = 
      \begin{pmatrix}
        -u(u-\xi_1)(u-1)\\
        0
      \end{pmatrix}.
\end{aligned}
\end{equation}
In the following, setting $X = \left({u \atop w}\right)$, we rewrite
equation (\ref{eq:mot_ex}) as 
\begin{equation}\label{eq:abstract}
  \begin{aligned}
    {\rm d}X(t)&= (AX(t)+F(X(t))) \, {\rm d}t + \sqrt{Q} \, {\rm
      d}W(t)
    \\
    X(0)&=x\in H
  \end{aligned}
\end{equation}
where $W(t)=(w_1(t), w_2(t))$ is a cylindrical Wiener process
on $H$ and $Q$ is the operator matrix
\begin{equation*}
  Q = \begin{pmatrix}
    Q_1 & 0\\
    0 & Q_2\end{pmatrix}.
\end{equation*}

Our first result is an existence and uniqueness theorem for the
solution of equation (\ref{eq:abstract}).

\begin{theorem}\label{thm:ex-intro}
  Let $x\in D(F)$ (resp. $x\in H$). Then, under the assumptions in
  Hypothesis \ref{hp:1}, there exists a unique mild (resp.
  generalized) solution 
  \begin{equation*}
    X \in L^2_W(\Om;C([0,T];H)) \cap L^2_W(\Om;L^2([0,T];V))
  \end{equation*}
  to equation \eqref{eq:abstract} which depends continuously on the
  initial condition.
\end{theorem}

The proof will be given in Section \ref{sez:4}. Starting from this
result, we can introduce the transition semigroup $P_t : C_b(H) \to
C_b(H)$ associated to the flow $X(t,\cdot)$ defined in equation
\eqref{eq:abstract}, that is
\begin{equation}
  P_t \phi(x)=\EE \phi(X(t,x)),\quad \phi \in C_b(H), t\geq 0, x\in H.
\end{equation}
In Theorem~\ref{thm:mis_inv} we shall prove the existence of an
invariant measure for $P_t$. After that, we shall prove that the
associated Kolmogorov operator $N_0$ is dissipative in the space
$L^p(H; \mu)$ and that its closure is $m$-dissipative.
    
\section{Preliminary results}
\label{sec:diss}

Before we proceed with the analysis of the abstract stochastic
equation (\ref{eq:abstract}) it is necessary to study the properties
of the operators $A$ and $F$. 

\begin{lemma}\label{dissF}
  Set $\eta=\tfrac{1}{3}(\xi_1^2-\xi_1+1)$; then 
  \begin{equation*}
    F_\eta \left({u \atop w}\right) = \left({f(u) - \eta u \atop
    0}\right)
  \end{equation*}
  is $m$-dissipative, that is, it is dissipative and $I - F_\eta$ maps
  $D(F)$ onto $H$, i.e., ${\mathrm{Rg}}(I - F_\eta) = H$.
\end{lemma}

\begin{proof}
  Let $x = \left({u \atop w}\right)$, $y = \left({v \atop z}\right)
  \in H$.  By definition,
  \begin{multline*}
    \langle F(x)-F(y) - \eta(x-y), x - y \rangle_H
    \\
    = \gamma \left\langle f(u)-f(v)-\eta(u-v), u-v
    \right\rangle_{L^2}
    \leq \, \gamma \left(\sup_{r\in \R} f'(r) - \eta
    \right)\left|u-v\right|^2_{L^2}.
  \end{multline*}
  We note that $\sup\limits_{r\in \R} f'(r) =
  \tfrac{1}{3}(\xi_1^2-\xi_1+1)$; thus the last term in the previous
  inequality vanishes and $F_\eta$ is dissipative.
  
  Let us show that $I - F_\eta$ is surjective. In fact, observe that
  its first component $-f(u)+(\eta+1)u$ a polynomial of degree 3 with
  positive derivative. Hence it is invertible. Its second component is
  the identity and, obviously, invertible. This concludes the proof.
\end{proof}

\begin{remark}
  We denote $f_\eta$ the first component of $F_\eta$. Setting $\xi_0
  =(1+\xi_1)/3$, $f_{\eta}(u)=f(u)-\eta u$ can be rewritten as
  \begin{equation*}
    f_{\eta}(u)=-(u-\xi_0)^3 - \xi_0^3.
  \end{equation*}
\end{remark}

Let us introduce the notation $A_\eta=A+\eta \left({I \atop 0}\right)$
and $F_{\eta}$ as above; then we may rewrite equation
\eqref{eq:abstract} as
\begin{equation}\label{eq:abstract2}
  \begin{aligned}
    {\rm d}X(t)&= (A_\eta X(t)+F_\eta(X(t))) \, {\rm d}t + \sqrt{Q} \,
    {\rm d}W(t)
    \\
    X(0)&=x\in H
  \end{aligned}
\end{equation}

\begin{lemma}\label{prop:dissA}
  $A_{\eta}$ is $m$-dissipative
  and in particular, there exist $\omega_1,\omega_2>0$ such that
  \begin{align}
    &\left\langle A_{\eta} x,x\right\rangle \leq -\omega_1 |x|^2_H
    \label{dissAH}
    \\
    &\left\langle A_{\eta} x,x\right\rangle \leq -\omega_2 \|x\|^2_V
    \label{dissAV}.
  \end{align}
\end{lemma}

\begin{proof}
  First of all, we observe that the operator $A_0$ satisfies the
  inequality:
  \begin{equation*}
    \langle \partial_\xi(c \, \partial_\xi u), u \rangle_{L^2} \leq 0.
  \end{equation*}
  In fact, with $c=\min_{[0,1]} c(\xi)$, we have:
  \begin{multline*}
    \int_0^1 \partial_\xi(c(\xi) \, \partial_\xi u(\xi)) u(\xi) \,
    {\rm d}\xi
    \\
    = c(\xi)u(\xi)\partial_{\xi}u(\xi) \Big|_0^1 - \int_0^1
    c(\xi)(\partial_{\xi}u(\xi))^2 \, {\rm d}\xi \leq
    -c\left|Du\right|^2_{L^2}\leq 0.
  \end{multline*}
  Now set $p=\min\limits_{[0,1]}p(\xi)>0$ and $\omega_1=\min \left\{ p-\eta,
    \alpha \right\}$. For $x = \left({u \atop v}\right)$ we have
  \begin{multline*}
    \left\langle Ax,x\right\rangle
    \le \gamma \langle \partial_\xi(c \, \partial_\xi u), u
    \rangle_{L^2} - \gamma (p-\eta)\left|u\right|^2_{L^2} - \gamma \langle u,
    v \rangle + \gamma \langle u,v \rangle - \alpha |v|^2_{L^2}
    \\
    \leq - \gamma (p-\eta) |u|^2_{L^2} -
    \alpha |v|_{L^2}^2 \le -\omega_1 \left|x\right|^2_H.
  \end{multline*}
  This proves \eqref{dissAH}.
      
  As \eqref{dissAV} is concerned, we have
  \begin{equation*}
    \left\langle Ax,x\right\rangle \leq -c
    \gamma\left|Du\right|^2_{L^2} - \gamma (p-\eta)\left|u\right|^2_{L^2}
    -\alpha\left|v\right|^2_{L^2} \leq -\omega_2
    (\gamma\left|u\right|_{H^1}+\left|v\right|^2)=-\omega_2
    \left\|x\right\|_V^2
  \end{equation*}
  for $\omega_2=\min\left\{c,p-\eta,\alpha \right\}$.
  
  Now let us show the $m$-dissipativity.We need to prove that 
  $I-A_{\eta}$ is surjective. Fix $x_0=(u_0,v_0)\in H$ and let we consider the following equation 
  \begin{equation*}
     \begin{cases}
         u-A_0u+v=u_0\\
         v-\gamma u+\alpha v=v_0.
     \end{cases}
  \end{equation*}
  Note that the second equality can be rewritten as 
  \begin{align}\label{eq:m-dissA}
      v=\dfrac{1}{1+\alpha}v_0+\dfrac{\gamma}{1+\alpha}u;
  \end{align}
  then, substituting $v$ with the right member of \eqref{eq:m-dissA} we obtain
  \begin{equation*}
      \left[\left(1-\dfrac{\gamma}{1+\gamma}\right)I-A_0\right]u=u_0-\dfrac{1}{1+\alpha}v_0.
  \end{equation*}
  Using the $m$-dissipativity of $A_0$ and (see for instance
  \cite{ouhabaz}) we obtain that previous equation admits a solution
  $u\in L^2(0,1)$.  We can then compute $v$ by means of
  \eqref{eq:m-dissA}.  It follows that for every $x_0$ there exists
  $x=(u,v)$ such that $(I-A_\eta)x = x_0$, that is $A_\eta$ is
  $m$-dissipative.
\end{proof}

From the above result it follows that $A_\eta$ is the infinitesimal
generator of a $C_0$ semigroup of contractions. Further, the following
holds.

\begin{proposition}\label{prop:genA}
  $A_\eta$ generates an analytic $C_0$-semigroup of contractions
  $e^{tA_\eta}$ on $H$ and it is of negative type.
\end{proposition}
    
\begin{proof}
  Note that $A_0$ and $-\alpha I$ generate analytic semigroups on
  $L^2(0,1)$ while $\gamma I$ is a bounded linear operator on the same
  space. Thus, the proof easily follows by applying the results in
  \cite[Section 4]{nagel/1989}.  Moreover, the dissipativity condition (\ref{eq:1})
  implies that $\left\|e^{tA_{\eta}}\right\|\leq e^{-\omega t}$, that
  is, $A_{\eta}$ is of negative type.
\end{proof}


For the moment, we notice that 
from the above lemmata 
we obtain the dissipativity of the sum $A_\eta + F_\eta$.

\begin{lemma}\label{le:3}
  Recall assumption (\ref{eq:1}), that we can write as
  \begin{equation}\label{hp_costanti}
    p-\eta \geq 0
  \end{equation} 
  where $p=\min\limits_{[0,1)} |p(\xi)|$. Then $A_\eta + F_\eta = A +
  F$ is dissipative.
\end{lemma}

\begin{proof}
  Observe that 
  \begin{multline*}
    \left\langle (A+F)x, x \right\rangle = \langle (A_{\eta}+F_\eta) x, x \rangle_H \leq - \gamma p |u|^2_{L^2} - \alpha
    |v|^2_{L^2} + \gamma \eta |u|^2_{L^2}
    \\
    \leq -\min\left\{p-\eta,\alpha\right\} |x|^2_H;
  \end{multline*}
  thus the dissipativity condition is satisfied if $p \geq \eta$.
\end{proof}

Setting $\omega=\min\left\{p-\eta,\alpha\right\}$, the statement of
Lemma \ref{le:3} can be rewritten as
  \begin{equation*}
    \langle (A+F)x, x \rangle_H \leq -\omega |x|^2_H.
  \end{equation*}

\subsection{An approximating problem}

In this section we show an existence and uniqueness result for a
family of approximating problems of system \eqref{eq:abstract2} with a
Lipschitz continuous nonlinearity.  Consider, for any $\ep >0$, the
following approximation of $F_{\eta}$, $F_{\eta,\ep}$, given as
\begin{equation*}
  F_{\eta,\ep}\left({u \atop v}\right) = \left({ f_{\eta,\ep}(u) \atop
      0} \right), \qquad f_{\eta,\ep}(u) = \dfrac{f(u)-\eta
    u}{1+\ep(1-\xi_0(u-\xi_0)+ (u-\xi_0)^2)}.
\end{equation*}
It is easily seen that $F_{\eta,\ep}$ is Lipschitz continuous and
\begin{equation*}
  \left|F_{\eta,\ep}(x)-F_{\eta}(x)\right|_{H} \to 0, \quad x\in
  L^6(0,1) \times L^2(0,1)
\end{equation*}
when $\ep \to 0$. 
Moreover it easy to see that
\begin{equation}\label{eq:crescita-pol-F_eta}
  \left|F_{\eta}(x)\right|_H \leq C(1+\left|x\right|_H^3), \quad x\in D(F),
\end{equation}
for suitable $C>0$.
    
Hence, for $\ep>0$, we are concerned with the family of equations
\begin{equation}\label{eq:approx}
  \begin{aligned}
    {\rm d}X(t)&= (A_\eta X(t)+F_{\eta,\ep}(X(t))) \, {\rm d}t +
    \sqrt{Q} \, {\rm d}W(t)
    \\
    X(0)&=x\in H
  \end{aligned}
\end{equation}
which can be seen as an approximating problem of \eqref{eq:abstract}.

There exists a well established theory on stochastic evolution
equations in Hilbert spaces, see Da Prato and Zabcyck
\cite{dpz:stochastic}, that we shall apply in order to show that for
any $\ep>0$ Equation (\ref{eq:approx}) admits a unique solution
$X_{\ep}(t)$. Let us recall from Proposition \ref{prop:genA} that
$A_{\eta}$ is the infinitesimal generator of a strongly continuous
semigroup $e^{tA_{\eta}}$, $t\geq 0$, on $H$; we also claim that the
following inequality hold:
%
%
\begin{equation}\label{eq:traccia}
  \int_0^t {\rm Tr}[e^{sA_{\eta}}Qe^{sA^*_{\eta}}] \, {\rm d}s
  <\infty, \quad \forall t\geq 0.
\end{equation}
(see below). If these properties are satisfied,
then the so-called stochastic convolution process
\begin{equation*}
  W_{A_\eta}(t)=\int_0^t e^{(t-s)A_{\eta}} \sqrt{Q} \, {\rm d}W(t)
\end{equation*}
is a well-defined mean square continuous, $\mathcal{F}_t$-adapted
Gaussian process (see \cite[Theorem 5.2]{dpz:stochastic}) and we can
give the following

\begin{definition}
  Given a $\mathcal{F}_t$-adapted cylindrical Wiener process on
  probability space $(\Omega, \mathcal{F},
  \left\{\mathcal{F}_t\right\}, \mathbb{P})$ a process $X(t)$, $t\geq
  0$, is a mild solution of \eqref{eq:approx} if it satisfies
  ${\mathbb P}$-a.s. the following integral equation
  \begin{equation}\label{eq:int}
    X(t)= e^{tA_{\eta}}x + \int_0^t e^{(t-s)A_{\eta}} F_{\eta,\ep}(X(s))
    \, {\rm d}s + \int_0^t e^{(t-s)A_{\eta}} \sqrt{Q} \, {\rm d}W(t).
  \end{equation}
\end{definition} 

Let us check that in our assumptions, condition \eqref{eq:traccia}
holds.

\begin{proposition}\label{prop:traccia}
  $A_{\eta}$ and $Q$ satisfy the following inequality:
  \begin{equation*}
    \sup_{t \ge 0} \int_0^{t} {\rm Tr}[e^{sA_{\eta}}Qe^{sA^*_{\eta}}]
    \, {\rm d}s < \infty.
  \end{equation*}
\end{proposition}

\begin{proof}
  Recall that if $S,T$ are linear operators
  defined on an Hilbert space H such that $S\in \mathcal{L}(H)$ and
  $T$ is of trace class, then
  \begin{equation}\label{dis:traccia}
    {\rm Tr}(ST) = {\rm Tr}(TS) \leq \left\|S\right\|_{\mathcal{L}(H)}
    {\rm Tr}(T). 
  \end{equation}
  Taking into account the self-adjointness of $A_{\eta}$ and the above
  remark we obtain
  \begin{equation*}
    {\rm Tr}[e^{tA_{\eta}}Qe^{tA^*_{\eta}}]\leq {\rm Tr}(Q)
    \left\|e^{tA_{\eta}}\right\|^2_{\mathcal{L}(H)} \leq {\rm Tr }(Q)
    e^{-2\omega t},
  \end{equation*}
  hence
  \begin{equation*}
    \int_0^{\infty} {\rm Tr}[e^{sA_{\eta}}Qe^{sA^*_{\eta}}] \, {\rm
      d}s \leq \int_0^{\infty} {\rm Tr}(Q) \, e^{-2\omega s} \, {\rm d}s <
    \infty.
  \end{equation*}
\end{proof}
    
\begin{proposition}
\label{pr_continuity_of_W_A}
  The stochastic convolution is $\P$-almost surely continuous on
  $[0,\infty)$ and it verifies the following estimate
  \begin{align}\label{eq:stimaW1}
    \EE \sup_{t \geq 0} \left|W_{A_{\eta}}(t)\right|^{2m}_H \leq C
  \end{align} 
  for some positive constant $C$.
\end{proposition}

\begin{proof}
  Note that for any $\alpha \in (0,1)$ it holds
  \begin{align*}
    \int_0^{\infty} s^{-\alpha} {\rm Tr}[e^{sA_{\eta}}Q
    e^{sA_{\eta}^*}] \, {\rm d}s <\infty.
  \end{align*}
  In fact,
  \begin{multline*}
    \int_0^{\infty} s^{-\alpha} {\rm Tr}[e^{sA_{\eta}}Q
    e^{sA_{\eta}^*}] \, {\rm d}s \leq {\rm Tr}(Q) \int_0^{\infty}
    s^{-\alpha} \left\|e^{sA_{\eta}}\right\|^2_{\mathcal{L}(H)} \,
    {\rm d}s
    \\
    \leq {\rm Tr}(Q) \, \int_0^{\infty} s^{-\alpha}e^{-2\omega_1 s} \,
    {\rm d}s <\infty.
  \end{multline*}
  Now the thesis follows by (\cite[Theorem 5.2.6]{dpz:ergodicity}.
\end{proof}
    
\begin{definition}
  Let $L^2_W(\Om; C([0,T];H))$ denote the Banach space of all
  $\mathcal{F}_t$-measurable, pathwise continuous processes, taking
  values in $H$, endowed with the norm
  \begin{equation*}
    \left\|X\right\|_{L^2_W(\Om;C([0,T];H))}=\left(\EE \sup_{t\in
      [0,T]}\left|X(t)\right|_H^2 \right)^{1/2}
  \end{equation*}
  while $L^2_W(\Om;L^2([0,T];V))$ denotes the Banach space of all
  mappings $X: [0,T] \to V$ such that $X(t)$ is
  $\mathcal{F}_t$-measurable, endowed with the norm
  \begin{equation*}
    \left\|X\right\|_{L^2_W(\Om;L^2([0,T];V))}=\left(\EE \int_0^T
      \left\|X(t)\right\|_V^2 \, {\rm d}t \right)^{1/2}.
  \end{equation*}
\end{definition}

With the above notation, Proposition \ref{pr_continuity_of_W_A}
implies that $W_A(t) \in L^2(\Omega; C([0,T];H))$ for arbitrary $T >
0$. Also, from Propositions \ref{prop:genA} and \ref{prop:traccia} it
follows that for $\ep > 0$ the approximating problems admit a unique
solution.

\begin{proposition}\label{prop:sol_approx}
  Let $x\in H$. Then, for any $\ep>0$ there exist a unique mild
  solution $X_{\ep}(t,x)$ to equation \eqref{eq:approx} such that
  \begin{equation*}
    X_{\ep}\in L^2_W(\Om;C([0,T];H))\cap L^2_W(\Om;L^2([0,T];V)).
  \end{equation*}
\end{proposition}

\begin{proof} 
  From \cite[Theorem 7.4]{dpz:stochastic} we have that for any $x\in
  H$ problem \eqref{eq:approx} has a unique mild solution
  $X_{\ep}(t,x)$ such that 
  \begin{equation*}
    \EE \sup_{t\in[0,T]} \left|X_{\ep}(t,x)\right|^p_H
    <C(1+\left|x\right|^p), \qquad p > 2,
  \end{equation*}
  which further admits a continuous modification; this proves that
  $X_{\ep} \in L^2_W(\Omega; C([0,T];H))$.  Now, we apply Ito's formula to
  the function $\phi(x)=\left|x\right|^2$ (although this is only
  formal, the following computations can be justified via a truncation
  argument) and we find that
  \begin{multline}\label{eq:Ito_app}
    \left|X_{\ep}(t,x)\right|^2 = \left|x\right|^2 + 2 \int_0^t
    \left\langle A_{\eta}X_{\ep}(s,x) + F_{\eta,\ep}(X(s,x)),
      X_{\ep}(s,x) \right\rangle \, {\rm d}s
    \\
    + 2 \int_0^t \left\langle X_{\ep}(s,x),\sqrt{Q} \, {\rm d}W(s)
    \right\rangle +{\rm Tr}(Q) \, t,
  \end{multline}
  where 
  \begin{equation*}
      \int_0^t \langle X_{\ep}(s,x),\sqrt{Q} \, {\rm d}W(s) \rangle
  \end{equation*}
  is a square integrable martingale such that, by \cite[Theorems 3.14
  and 4.12]{dpz:stochastic}, 
  \begin{equation*}
    \EE \sup_{t \in [0,T]} \left|\int_0^t \left\langle
        X_{\ep}(s,x),\sqrt{Q} \, {\rm d}W(s)\right\rangle \right| \leq
    3 {\rm Tr}(Q) \, 
    \EE \left(\int_0^T
      \left|X_{\ep}(s,x)\right|^2_H \, {\rm d}s\right).
  \end{equation*}
  Moreover we have
  \begin{equation*}
    \int_0^t \langle A_{\eta}X_{\ep}(s,x), X(s,x) \rangle \, {\rm
      d}s \leq - \omega_2 \int_0^{t} \|X_{\ep}(s,x)\|^2_V \, {\rm
      d}s
  \end{equation*}
  and
  \begin{align*}
    \int_0^t \langle F_{\eta,\ep}(X_{\ep}(s,x)), X_{\eta}(s,x) \rangle
    \, {\rm d}s \leq 0.
  \end{align*}
  Hence, taking the expectation of both members in \eqref{eq:Ito_app}
  we obtain
  \begin{equation*}
    \EE \sup_{t \in [0,T]} \left|X(t,x)\right|^2 + \omega_2 \EE
    \int_0^T \left\|X_{\ep}(s,x)\right\|^2_V \, {\rm d}s \leq
    \left|x\right|^2 + (6+\eta)\int_0^T \EE \sup_{s \in [0,t]}
    \left|X_{\ep}(s,x)\right|^2_H \, {\rm d}t.
  \end{equation*}
  By Gronwall's lemma this yields
  \begin{equation}\label{stima_lim}
    \EE \sup_{t \in [0,T]} \left|X_{\ep}(t,x)\right|^2_H + \omega_2
    \EE \int_0^T \left\|X_{\ep}(s,x)\right\|^2_V \, {\rm d}s \leq
    C(\left|x\right|^2+1).
  \end{equation}
  We conclude that $X_{\ep}\in L^2(\Om;L^2([0,T]; V))$.
\end{proof}

\section{Existence and uniqueness result}
\label{sez:4}

Here we make use of the results given in the last section to show that
problem \eqref{eq:abstract2} admits a unique solution.  Our main
result can be stated as follows.

\begin{theorem}\label{thm:ex}
  For every $x\in D(F)$ (resp. $x\in H$), there exists a unique mild
  (resp.\  generalized) solution $X\in L^2_W(\Om;C([0,T];H))\cap
  L^2_W(\Om;L^2([0,T];V)) $ to equation \eqref{eq:abstract} which
  satisfies
  \begin{equation}\label{eq:dip_d_iniz}
    \EE \left|X(t,x)-X(t,\bar{x})\right|^2_H \leq C
    \left|x-\bar{x}\right|^2_H.
  \end{equation}
\end{theorem}
    
\begin{proof}
  As shown in the proof of Proposition \ref{prop:sol_approx},
  $\left\{X_{\ep}\right\}_{\ep \geq 0}$ satisfies
  \begin{equation*}
    \EE \sup_{t\in [0,T]} \left|X_{\ep}(t,x)\right|^2_H + \omega_1 \EE
    \int_0^t \left\|X_{\ep}(s,x)\right\|^2_V \, {\rm d}s \leq
    C(\left|x\right|^2+1), \qquad t\geq 0.
  \end{equation*}
  therefore it is bounded in $L^2_W(\Om;C([0,T];H))\cap
  L^2_W(\Om;L^2([0,T];V))$.

  We are going to show the following estimates
  \begin{align}
    &\EE \int_0^T \left|f_{\eta,\ep}(X_\ep(t,x))\right|^2_H \, {\rm
      d}t \leq C,
    \label{stima_f_ep}
    \\
    &\EE \sup_{t\in [0,T]}
    \left|X_{\ep}(t,x)-X_{\lb}(t,x)\right|^2_{H} \leq C(\lb+\ep),
    \label{stimaC^0}
  \end{align}
  where we use the same symbol $C$ to denote several positive
  constants independent of $\ep$.
  Using the above results, we conclude that
  $\left\{X_{\ep}\right\}_{\ep}$ is a Cauchy sequence on
  \begin{equation*}
    L^2(\Om;C([0,T];H))\cap L^2(\Om;L^2([0,T];V))
  \end{equation*}
  and, consequently, it converges uniformly on $[0,T]$ to a process
  $X(t,x)$.
    
  {\em Step 1.} We begin with the continuous dependence on the initial
  condition. Let us consider the difference
  $X_{\ep}(t,x)-X_{\ep}(t,\bar{x})$, for $x, \bar{x} \in H$.

  Note that 
  \begin{multline*}
    {\rm d}X_{\ep}(t,x) - {\rm d}X_{\ep}(t,\bar{x})
    \\
    = A_{\eta} \left[X_{\ep}(t,x)-X_{\ep}(t,\bar{x})\right] \, {\rm d}t +
    \left[F_{\eta,\ep}(X_{\ep}(t,x))-F_{\eta,\ep}(X_{\ep}(t,\bar{x}))\right]
    \, {\rm d}t
  \end{multline*}
  hence
  \begin{align*}
    |X_{\ep}(t,x) &- X_{\ep}(t,\bar{x})|^2_H
    \\
    &= \left|x-\bar{x}\right|^2+2\int_0^t \left\langle A_{\eta}(
      X_{\ep}(s,x)-X_{\ep}(s,\bar{x})),X_{\ep}(s,x)-X_{\ep}(s,\bar{x})\right\rangle
    \, {\rm d}s
    \\
    &\quad + 2\int_0^t \left\langle F_{\eta,\ep}(X_{\ep}(s,x)) -
      F_{\eta,\ep}(X_{\ep}(s,\bar{x})), X_{\ep}(s,x)-X_{\ep}(s,\bar{x})
    \right\rangle \, {\rm d}s
  \end{align*}
  and therefore
  \begin{equation*}
    \EE\left|X_{\ep}(t,x)-X_{\ep}(t,\bar{x})\right|^2_H \leq
    \EE\left|x-\bar{x}\right|^2_H - 2\omega \int_0^t
    \EE\left|X_{\ep}(s,x)-X_{\ep}(s,\bar{x})\right|^2_H \, {\rm d}s.
  \end{equation*}
  Applying Gronwall's lemma we obtain
  \begin{equation}\label{eq:dip_dato_iniz}
    \EE\left|X_{\ep}(t,x)-X_{\ep}(t,\bar{x})\right|^2_H \leq e^{-2
      \omega t}\left|x-\bar{x}\right|^2_H.
  \end{equation}

  The continuity condition (\ref{eq:dip_d_iniz}) easily implies
  uniqueness of the mild solution on $D(F)$ and of the generalized
  solution on $H$. Consequently, it only remains to prove existence.

  {\em Step 2.}  Next, let us consider estimate \eqref{stima_f_ep}.
  We shall apply Ito's formula to the function
  \begin{equation*}
    \phi(x) = \int_0^1 g_{\ep}(u(\xi)) \, {\rm d}\xi, \qquad
    x=(u(\xi),v(\xi)) \in H,
  \end{equation*}
  where
  \begin{equation*}
    g_{\ep}(r) = - \int_0^r f_{\eta,\ep}(s) \, {\rm d}s, \qquad r\in \R^+,\
    \ep>0.
  \end{equation*}
  It is not difficult to show that, for any $x \in D(F)$,
  \begin{equation*}
    D\phi(x)= 
    \begin{pmatrix}
      -f_{\eta,\ep}(u) \\ 0
    \end{pmatrix}
    \qquad \textrm{and} \qquad D^2\phi(x)= 
    \begin{pmatrix}
      -f_{\eta,\ep}'(u) & 0 \\ 0 & 0
    \end{pmatrix},
  \end{equation*}
  thus
  \begin{multline*}
    \left\langle A_{\eta}X_{\ep} + F_{\eta,\ep}(X_\ep),
      D\phi(X_{\ep})\right\rangle = -\gamma \left\langle \partial_{\xi}
      (c(\cdot) \partial_\xi U_{\ep}), f_{\eta,\ep}(U_{\ep}) \right\rangle
    \\
    + \gamma \left\langle
      (p(\xi)-\eta)U_{\ep},f_{\eta,\ep}(U_{\ep})\right\rangle +\gamma
    \left\langle V_{\ep},f_{\eta,\ep}(U_{\ep})\right\rangle-\gamma
    \left|f_{\eta,\ep}(U_{\ep})\right|^2.
  \end{multline*}
  We claim that 
  \begin{multline}\label{eq:f'ep}      
    f_{\eta,\ep}'(u) = -\ep \, \frac{\left( -1 + 2\,\left( u - \xi_0
        \right) \right) \, \left( -{\left( u - \xi_0 \right) }^3 -
        {\xi_0 }^3 \right) }{{\left( 1 + \ep \,\left( 1 - u + {\left(
                u - \xi_0 \right) }^2 + \xi_0 \right) \right) }^2}
    \\
    - \frac{3\,{\left( u - \xi_0 \right) }^2} {1 + \ep \,\left( 1 - u
        + {\left( u - \xi_0 \right) }^2 + \xi_0 \right) }
  \end{multline}
  is always negative; then it follows that
  \begin{equation*}
    -\int_0^1 (\partial_{\xi} c(\xi)\partial_{\xi} u)
    f_{\eta,\ep}(u) \, {\rm d}\xi =
    - c(\xi)\partial_{\xi}u(\xi)\Big|_0^1 - \int_0^1
    c(\xi)(\partial_{\xi} u(\xi))^2 f'_{\ep}(u(\xi)) \, {\rm d}\xi \le 0
  \end{equation*}
  and, for any $\sigma >0$
  \begin{equation*}
    \left\langle v, f_{\eta,\ep}(u) \right\rangle \leq \sigma
    \left|v\right|^2 + \frac{1}{\sigma} \left|f_{\eta,\ep}(u)\right|^2.
  \end{equation*}
  From the above inequalities it follows that for $\sigma$
  sufficiently large
  \begin{equation} \label{AFphi}
    \begin{aligned}
      \langle A_{\eta}X_{\ep} &+ F_{\eta,\ep}(X_\ep), D\phi(X_{\ep})\rangle 
      \\
      &\leq \gamma \sigma \left(\eta + \|p\|^2_{L^\infty([0,1])}\right) |U_\ep|^2
      + \gamma \sigma \left|V_{\ep}\right|^2
      +\gamma\left(\dfrac{2}{\sigma}-1\right)
      \left|f_{\eta,\ep}(U_{\ep})\right|^2 
      \\
      &\leq - C \left|f_{\eta,\ep}(U_{\ep})\right|^2 +
      K\left|X_{\ep}\right|^2
    \end{aligned}
  \end{equation}
  for suitable constants $C,K$.  Further,
  \begin{equation*}
    {\rm Tr} (Q D^2\phi(X_{\ep}))
    = -\sum_{k=1}^{\infty} \left\langle Q_1
      f_{\eta,\ep}'(U_{\ep})e_{k},e_k\right\rangle
    = - \sum_{k=1}^{\infty} \lambda_k \int_0^1
      f_{\eta,\ep}'(U_{\ep}(\xi)) e_k^2(\xi) \, {\rm d}\xi.
  \end{equation*}

  Now we observe that
  \begin{align*}
    \frac{\left|-3(u-\xi_0)^2 -
        \ep(-\xi_0+2(u-\xi_0))(-(u-\xi_0)^3-\xi_0^3)\right|}{1 + \ep -
      \xi_0 \ep(u-\xi_0) + \ep(u-\xi_0)^2} \leq
    4\left(\left|u-\xi_0\right|^2+\ep\right),
  \end{align*}
  so that for $\ep$ sufficiently small, taking into account
  \eqref{eq:f'ep} and the uniform bound condition on the $e_k$ stated
  in assumption \eqref{eq:stima_e_k}, we have
  \begin{multline*}
    \left|f'_{\ep}(u(\xi))e_k^2\right| \leq
    4\left(\left|u(\xi)-\xi_0\right|^2+\ep\right)\dfrac{\left|e_k^2(\xi)\right|}{1+\ep-\xi_0\ep(u(\xi)-\xi_0)+\ep(u(\xi)-\xi_0)^2}
    \\ \leq C\left(\left|u(\xi)-\xi_0\right|^2+\ep\right) \leq
    C\left(\left|u(\xi)\right|^2+1\right),
  \end{multline*}
  therefore
  \begin{multline} \label{TrQphi}
    \EE \int_0^t \left|{\rm Tr}\left[Q D^2\phi(X_{\ep}(s))\right]
    \right| \, {\rm d}s \leq \EE \int_0^t {\rm d}s \left(\int_0^1
      \sum_{k=1}^{\infty} \lambda_k \left|f_{\eta,\ep}'(U_{\ep}(s)(\xi))
        e_k^2(\xi)\right| \, {\rm d}\xi \right)
    \\
    \leq \, C \left(1 + \EE \int_0^t \left|X_{\ep}(s)\right|^2_H \,
      {\rm d}s \right).
 \end{multline}
  Estimates \eqref{AFphi} and \eqref{TrQphi} yield
  \begin{multline*}
    \EE \phi(X_{\ep}(t,x)) +  \EE \int_0^t
    \left|f_{\eta,\ep}(X_{\ep}(s,x))\right|^2 \, {\rm d}s 
    \\ 
    \leq \phi(x)+C\left(1+\EE\int_0^t \left|X_{\ep}(s,x)\right|^2 {\rm d}s \right) \leq 
    C, 
  \end{multline*}
  and therefore
  \begin{equation*}
    \EE \int_0^T \left|f_{\eta,\ep}(X_{\ep}(t,x))\right|^2 \, {\rm d}t \leq C,
  \end{equation*}
  so that inequality \eqref{stima_f_ep} is proved.
  
  {\em Step 3.} We proceed to estimate \eqref{stimaC^0}. We observe
  that
  \begin{equation*}
    {\rm d}(X_{\lb}(t,x)-X_{\ep}(t,x)) =
    \left[A_{\eta}(X_{\lb}(t,x)-X_{\ep}(t,x)) +
      F_{\eta,\lb}(X_{\lb}(t,x))-F_{\eta,\ep}(X_{\ep}(s,x))\right] \, {\rm d}t.
  \end{equation*}
  Hence, using Ito's formula as before we get
  \begin{equation*}\label{f_ep-f_lb1}
    \begin{aligned}
      \EE \sup_{[0,T]}|X_{\lb}(t,x) &- X_{\ep}(t,x)|^2
      \\
      =& \EE \int_0^T \left\langle
        A_{\eta}(X_{\lb}(s,x)-X_{\ep}(s,x)),X_{\lb}(s,x)-X_{\ep}(s,x)\right\rangle
      \, {\rm d}s
      \\
      &+ \EE \int_0^T \left\langle F_{\eta,\lb}(X_{\lb}(s,x) -
        F_{\eta,\ep}(X_{\ep}(s,x)), X_{\lb}(s,x)-X_{\ep}(s,x)\right\rangle
      \, {\rm d}s
      \\
      \leq& -\omega_2 \EE \int_0^T
      \left\|X_{\lb}(s,x)-X_{\ep}(s,x)\right\|^2 \, {\rm d}s
      \\
      &+ \EE\int_0^T \left\langle f_{\eta,\lb}(U_{\lb}(s,x)) -
        f_{\eta,\ep}(U_{\ep}(s,x)), U_\lb(s,x) - U_\ep(s,x)\right\rangle \,
      {\rm d}s.
  \end{aligned}    
  \end{equation*}
  Now set 
  \begin{align*}
    h_{\ep}(u)&=\dfrac{f_{\eta,\ep}(u)}{1+\ep
      -\ep\xi_0(u-\xi_0)+\ep(u-\xi_0)^2} + \dfrac{\xi_0^3}{1+\ep
      (1-\xi_0(u-\xi_0)+(u-\xi_0)^2)}
    \\
    &=\dfrac{-(u-\xi_0)^3}{1+\ep(1 -\xi_0(u-\xi_0)+(u-\xi_0)^2)}.
  \end{align*}
  We note that, for any $u,v$,
  \begin{equation} \label{f_ep-f_lb2}
    (h_{\lb}(u)-h_{\ep}(v))(u-v) \leq
    (h_{\lb}(u)-h_{\ep}(v))((u+h^{1/3}_{\lb}(u))-(v+h^{1/3}_{\ep}(v))).
  \end{equation}
  In fact, 
  \begin{multline*}
    (h_{\lb}(u)-h_{\ep}(v))(u-v) -
    (h_{\lb}(u)-h_{\ep}(v))((u+h^{1/3}_{\lb}(u))-(v+h^{1/3}_{\ep}(v)))
    \\
    =\,-(h^{1/3}_{\lb}(u)-h^{1/3}_{\ep}(v))^2
    (h^{2/3}_{\lb}(u)+h^{1/3}_{\ep}(u)h^{1/3}_{\lb}(u)+h^{2/3}_{\lb}(u))\leq
    0.
  \end{multline*}
  Moreover one can compute
  \begin{equation*}
    \left|u-\xi_0+h_{\ep}^{1/3}(u)\right|\leq \ep \left|h_{\ep}(u)\right|,
  \end{equation*}
  therefore
  \begin{align}\label{dis:geta}
    (h_{\ep}(u)-h_{\lb}(v))(u-v) &\leq
    \left(\left|h_{\ep}(u)\right|+\left|h_{\lb}(v)\right|\right)
    \left(\ep\left|
        h_{\ep}(u)\right|+\lb\left|h_{\lb}(v)\right|\right) 
    \notag
    \\
    & \leq
    C(\ep+\lb)(\left|h_{\ep}(u)\right|^2+\left|h_{\lb}(v)\right|^2).
  \end{align}
  Furthermore, we observe that
  \begin{align}\label{dis:f_xi_0}
    &\xi_0^3\left(-\dfrac{1}{1+\ep -\ep\xi_0(u-\xi_0)+\ep(u-\xi_0)^2}
      + \dfrac{1}{1+\ep -\ep\xi_0(v-\xi_0)+\ep(v-\xi_0)^2}\right)(u-v)
    \notag\\
    &=\xi_0^3\dfrac{-\ep(u-\xi_0)^2+\lb(v-\xi_0)^2}{(1+\ep
      -\ep\xi_0(u-\xi_0)+\ep(u-\xi_0)^2)(1+\ep
      -\ep\xi_0(v-\xi_0)+\ep(v-\xi_0)^2)}(u-v)
    \notag\\
    &\leq \xi_0^3(\ep+\lb)\left[\left|u-v\right| +
      \left|u-\xi_0\right|^2 + \left|v-\xi_0\right|^2 +
      \left|u-\xi_0\right|^3 + \left|v-\xi_0\right|^3\right]
  \end{align}
  Combining \eqref{dis:geta} and \eqref{dis:f_xi_0} we get 
  \begin{multline*}
    (f_{\eta,\ep}(u)-f_{\eta,\lb}(v))(u-v)\leq \\
    C(\ep+\lb)\left(\left|h_{\ep}(u)\right|^2 +
      \left|h_{\lb}(v)\right|^2 +\left|u-v\right| +
      \left|u-\xi_0\right|^2 + \left|v-\xi_0\right|^2 +
      \left|u-\xi_0\right|^3 + \left|v-\xi_0\right|^3\right)
  \end{multline*}
  and, consequently,
  \begin{multline*}
    \EE \sup_{[0,T]}\left|X_{\lb}(t,x)-X_{\ep}(t,x)\right|^2 
    \\
    \leq \EE\int_0^T \int_0^1
    (f_{\eta,\lb}(X_{\lb})-f_{\eta,\ep}(X_{\ep}))(X_{\lb}-X_{\lb}) \, {\rm d}\xi \, {\rm d}t \leq
    C(\ep+\lb)
  \end{multline*}
  We conclude that there exists the limit $X=\lim\limits_{\ep \to 0} X_{\ep}$
  in $L^2(\Omega; C([0,T];H))$ and, by \eqref{stima_lim}, also that
  $X\in L^2(\Omega;L^2([0,T]; V))$. Moreover, estimate
  \eqref{eq:dip_dato_iniz} implies inequality \eqref{eq:dip_d_iniz}.
\end{proof}

We conclude the section with another estimate which turns out to be
useful when we will deal with the asymptotic behaviour of the
solution.

\begin{lemma}\label{lemma:stima_potenza2m}
  The following estimate holds
  \begin{align*}
    \EE \left|X(t,x)\right|^{2m} \leq C_m \left(1+e^{-m\omega_1
        t}\left|x\right|^{2m}\right), \quad x\in H, t\geq 0.
  \end{align*}
\end{lemma}
             
\begin{proof}
  Let $Y(t)=X(t,x)-W_{A_\eta}(t)$. Then 
 \begin{equation*}
   \tfrac{{\rm d}}{{\rm d}t} Y(t)=A_{\eta}
   Y(t)+F_{\eta}(Y(t)+W_{A_{\eta}}(t)), \qquad Y(0)=x.
 \end{equation*}
 Observe that
 \begin{align*}
   \tfrac{1}{2m} \tfrac{{\rm d}}{{\rm d}t} \left|Y(t)\right|^{2m} &=
   \left|Y(t)\right|^{2m-2} \tfrac{{\rm d}}{{\rm d}t}
   \left|Y(t)\right|^2
   \\
   &\leq -\omega_1 \left|Y(t)\right|^{2m} + \left\langle
     F_{\eta}(Y(t)+W_{A_{\eta}}(t)), Y(t) \right\rangle
   \left|Y(t)\right|^{2m-2}
   \\
   & \leq -\omega_1\left|Y(t)\right|^{2m} + \left\langle
     F_{\eta}(W_{A_{\eta}}(t)),Y(t)\right\rangle
   \left|Y(t)\right|^{2m-2}
   \\
   & \leq -\omega_1+\left|F_{\eta}(W_{A_{\eta}}(t))\right|
   \left|Y(t)\right|^{2m-1}
   \end{align*}
   Hence
   we conclude that
   \begin{align*}
     \tfrac{1}{2m} \tfrac{{\rm d}}{{\rm d}t} \left|Y(t)\right|^{2m}
     \leq -\omega_1 \left|Y(t)\right|^{2m}+ C
     \left|F_{\eta}(W_{A_{\eta}}(t))\right|^{2m}.
 \end{align*}
 for some $C > 0$. By Gronwall's lemma it follows that
 \begin{equation*}
   \left|Y(t)\right|^{2m} \leq e^{-m\omega_1 t} \left|x\right|^{2m} +
   2m C \int_0^t e^{-m \omega_1(t-s)}\left|F_{\eta}(W_{A_{\eta}}(s))\right|^{2m} \,
   {\rm d}s,
 \end{equation*}
 so that for some $C > 0$ (possibly different from the above): 
 \begin{multline}\label{eq:potenza2m}
   \left|X(t,x)\right|^{2m} 
   \\
   \leq C\left( e^{-m\omega_1
     t}\left|x\right|^{2m}
   +\int_0^t e^{-m\omega_1
       (t-s)}\left|F_{\eta}(W_{A_{\eta}}(s))\right|^{2m} \, {\rm d}s +
     \left|W_{A_{\eta}}(t)\right|^{2m}\right).
 \end{multline}
 Now recall that $F_{\eta}$ has polynomial growth (see 
 \eqref{eq:crescita-pol-F_eta}); in particular we have that
 \begin{equation*}
   \left|F_{\eta}(W_{A_{\eta}}(t))\right|^{2m} \leq C \left(1 +
     \left|W_{A_{\eta}}(t)\right|^{3}\right)^{2m} \leq C (1 +
   \left|W_{A_{\eta}}(t)\right|^{6m}).
 \end{equation*}
 Moreover, by \eqref{eq:stimaW1}, $\sup\limits_{t\geq 0}\EE
 \left|W_{A_{\eta}}(t)\right|^{2m} < C_m$,
 then
 \begin{multline*}
   \int_0^t e^{-m\omega_1 (t-s)}\left|F_{\eta}(W_{A_{\eta}}(t))\right|^{2m} \, {\rm
     d}s \leq C \int_0^t e^{-m\omega_1 (t-s)} \left(1 +
     \left|W_{A_{\eta}}(s)\right|^{6m}\right) \, {\rm d}s
   \\
   \leq C \int_0^t e^{-m\omega_1 (t-s)} \left(1 + C_m^3\right) \,
   {\rm d}s \leq C'_m.
 \end{multline*}
 Using the last estimate in \eqref{eq:potenza2m} we conclude the
 proof.
 
\end{proof}
     
\section{Asymptotic behaviour of solutions}

Let $P_t : C_b(H) \to C_b(H)$ be the transition semigroup associated
to the flow $X(t,\cdot)$ defined in equation \eqref{eq:abstract}, that
is
\begin{equation}
  P_t \phi(x)=\EE \phi(X(t,x)),\quad \phi \in C_b(H), t\geq 0, x\in H.
\end{equation}
We are ready to prove the main result of the paper.
     
\begin{theorem}\label{thm:mis_inv}
  Under hypothesis \ref{hp:1} there exists a unique invariant measure
  $\mu$ for $P_t$.
\end{theorem}
     
\begin{proof}
  To discuss the existence of the invariant measure, it will be
  convenient to consider equation \eqref{eq:abstract} on the whole
  real line. Therefore we extend the process $W(t)$ for $t<0$ by
  choosing a process $\tilde{W}(t)$ with the same law as $W(t)$ but
  independent of it and setting
  \begin{equation*}
    W(t)=\tilde{W}(-t), \quad t\leq 0.
  \end{equation*}
  Now, for any $\lambda >0$, denote by $X_{\lb}(t,x)$, $t\geq -\lb$,
  the unique solution of
  \begin{align*}
    &{\rm d}X=[A_{\eta}X+F_{\eta}(X)] \, {\rm d}t+\sqrt{Q} \, {\rm d}W(t)\\
    &X(-\lb)=x\in H.
  \end{align*}
  Then $X_{\lb}$ satisfies the following integral equation: 
  \begin{align*}
    X_{\lb}(t,x) = x + \int_{-\lb}^t
    A_{\eta}X_{\lb}(s,x)+F(X_{\lb}(s,x)) \, {\rm d}s + \int_{-\lb}^t
    \sqrt{Q} \, {\rm d}W(s).
  \end{align*}
  We note that 
  \begin{align*}
    X(\lb,x) &= x + \int_{0}^{\lb}
    A_{\eta}X_{\lb}(s,x)+F(X_{\lb}(s,x)) \, {\rm d}s + \int_{0}^{\lb}
    \sqrt{Q} \, {\rm d}W(s)
    \\
    &= x + \int_{-\lb}^{0} A_{\eta}X_{\lb}(s,x)+F(X_{\lb}(s,x)) \,
    {\rm d}s + \int_{-\lb}^{0} \sqrt{Q} \, {\rm d}W(s)
    \\
    &= X_{\lb}(0,x).
  \end{align*}         
  Thus, the theorem will be proved once we establish that
  \begin{equation*}
    \lim_{\lb \to \infty} \mathcal{L}(X_{\lb}(0,x)) = \mu 
  \end{equation*}
  weakly, for some $\mu \in M^+_1(H)$ and all $x \in H$. As in
  \cite[Theorem 11.21]{dpz:stochastic} we will not prove only this,
  but we will show that there exists a random variable $Y\in
  L^2(\Omega;\mathcal{F},\P)$ such that
  \begin{equation}\label{eq:X_lb}
    \lim_{\lb \to \infty} \EE \left|X_{\lb}(t,x)-Y\right|^2=0, \quad x\in H,
  \end{equation}
  and the law of $Y$ is the required stationary distribution.
         
  We first prove that \eqref{eq:X_lb} is true when $x=0$. We put
  $X_{\lb}(t,0)=X_{\lb}(t)$.  Proceeding as in Theorem \ref{thm:ex} we
  obtain
  \begin{align*}
    \EE \left|X_{\lb}(t)\right|^2 \leq -2 \omega\EE\int_{-\lb}^t
    \left|X_{\lb}(s)\right|^2_H \, {\rm d}s + 2{\rm Tr}[Q]t
  \end{align*} 
  Using Gronwall's lemma, we have
  \begin{align}\label{eq:lim1}
    \EE \left|X_{\lb}(t)\right|^2 \leq (2{\rm
      Tr}[Q](t+\lb)+\left|x\right|^2)e^{-2\omega(t+\lb)}\leq C, \quad
    \forall \lb>0, \forall t\in [-\lambda,\infty].
  \end{align}
  
  We can now prove \eqref{eq:X_lb}. Let $\gamma <\lb$; then
  \begin{align*}
    X_{\lb}(t,0)=X_{\gamma}(t,X_{\lb}(-\gamma,0)), \quad t\geq -\gamma
  \end{align*}
  and, proceeding as in Theorem \ref{thm:ex} we obtain an estimate
  similar to that in \eqref{eq:dip_dato_iniz}
  \begin{multline}\label{eq:Cauchy1}
    \EE\left|X_{\lb}(t,0)-X_{\gamma}(t,0)\right|^2 =
    \EE\left|X_{\gamma}(t,X_{\lb}(-\gamma,0))-X_{\gamma}(t,0)\right|^2 
    \\
    \leq e^{-\omega_1(t+\gamma)}\EE\left|X_{\lb}(-\gamma)\right|^2_H
    \leq C e^{-\omega_1(t+\gamma)}.
  \end{multline} 
  Estimates \eqref{eq:lim1} and \eqref{eq:Cauchy1} imply that
  $\left\{X_{\lb}(0)\right\}_{\lb\geq 0}$ is a bounded Cauchy sequence
  in $L^2_W(\Omega;H)$.  Then there exists a random
  variable $Y$ such that $\EE\left|X_{\lb}(0)-Y\right|^2_H \to 0$, as
  $\lb \to \infty$. Proceeding similarly we show that
  \begin{equation*}
    \lim_{\lambda \to \infty}
    \EE\left|X_{\lb}(0,x)-X_{\lb}(0)\right|^2=0, \quad \forall x\in H.
  \end{equation*}
  This ends the proof.
\end{proof}

\begin{lemma}\label{lemma:uguaglianzaP_t}
  For any $\phi \in C_b(H)$ and $x\in H$ there exists the limit
  \begin{align*}
    \lim_{t \to \infty} P_t \phi=\int_H \phi(y)\mu({\rm d}y).
  \end{align*}
\end{lemma}
             
\begin{proof}
  Set $Y=\lim\limits_{s\to-\infty} X(0,-s,x)\in L^2(\Omega;H)$, which
  exists in virtue of Theorem \ref{thm:mis_inv}.  Then $$P_t
  \phi(x)=\EE\left[\phi(X(t,0,x))\right]=\EE
  \left[\phi(X(0,-t,x))\right].$$
  By the dominated convergence theorem
  it follows that
  \begin{align*}
    \lim_{t\to \infty} P_t\phi(x) =\EE \left[\phi (Y)\right]=\int_H
    \phi(y)\mu({\rm d}y).
  \end{align*}
\end{proof}
         
\section{The infinitesimal generator of $P_t$}

\begin{lemma}
  For any $p\geq 1$ $P_t$ has a unique extension to a strongly
  continuous semigroup of contraction in $L^p(H,\mu)$ which we still
  denote by $P_t$.
\end{lemma}

\begin{proof}
  Let $\phi \in C_b(H)$ and $\mu_t$ be the law of $X(t,x)$. By Holder
  inequality we have that
  \begin{align*}
    \left|P_t \phi(x)\right|^p \leq P_t \left|\phi(x)\right|^p.
  \end{align*}
  Integrating this identity with respect to $\mu$ over $H$ and taking
  into account the invariance of $\mu$, we obtain
  \begin{align*}
    \int_H \left|P_t \phi(x)\right|^p \mu({\rm d}x) \leq \int_H P_t
    \left|\phi\right|^p(x)\mu({\rm d}x)=\int_H \left|\phi(x)\right|^p
    \mu({\rm d}x).
  \end{align*}
  Since $C_b(H)$ is dense in $L^p(H,\mu)$, $P_t$ can be uniquely
  extended to a contraction semigroup in $L^p(H,\mu)$. The strong
  continuity of $P_t$ follows from the dominated convergence theorem.
\end{proof}
             
Taking into account the Hille-Yosida's theorem, from the previous
lemma we deduce that the infinitesimal generator of $P_t$ on
$L^p(H,\mu)$ (which we denote by $N$) is closed, densely defined and
it satisfies
\begin{equation*}
  \left|\lambda R(\lambda,A)\right| \leq 1.
\end{equation*} 
We want to show that $N$ is the closure of the differential operator
$N_0$ defined by
\begin{align*}
  N_0 \phi =\dfrac{1}{2} {\rm Tr}[QD^2 \phi(x)]+\left\langle
    x,AD\phi(x)\right\rangle + \left\langle F(x),D\phi(x)\right\rangle
\end{align*} 
on $\mathcal{E}_A(H)= {\rm linear\ span}\left\{\phi = e^{\left\langle
      x,h\right\rangle}\mid h\in D(A)\right\}$.
             
We recall that the operator 
\begin{equation*}
  L\phi =\dfrac{1}{2} {\rm Tr}[QD^2 \phi(x)] + \langle x, AD\phi(x)
  \rangle 
\end{equation*}
is the Ornstein-Uhlenbeck operator and it verifies
\begin{align}\label{eq:L-phi}
  \left|L\phi(x)\right| \leq a +b\left|x\right|, \quad x\in H.
\end{align}
(see \cite[Section 2.6]{dp:kolmogorov}).
             
Then, in order to show that $N_0$ is well-defined as an operator with
values in $L^p(H,\mu)$ we need that $F(x)\in L^p(H,\mu)$.  This is
provided by the following result.

\begin{lemma}
  Under Hypothesis \ref{hp:1}, there exists $c_m$ depending only on
  $A_{\eta}$ and $F_{\eta}$ such that
  \begin{equation}\label{eq:stima-integrale}
    \int_H \left|x\right|^{2m}_H \, \mu({\rm d}x) \leq c_m.
  \end{equation}
\end{lemma}
             
\begin{proof}
  Denote by $\mu_{t,x}$ the law of $X(t,x)$. Then by lemma
  \ref{lemma:stima_potenza2m} we have that for any $\beta>0$
  \begin{multline*}
    \int_H \dfrac{\left|y\right|^{2m}}{1+\beta \left|y\right|^{2m}}
    \mu_{t,x}({\rm d}y) \leq \int_H \left|y\right|^{2m} \mu_{t,x}({\rm
    d}y)
    \\
    = \EE \left|X(t,x)\right|^{2m} \leq C_m(1+e^{-m\omega_1
      t}\left|x\right|^{2m}), \quad x\in H.
  \end{multline*}
  Consequently, letting $t \to \infty$ we find, taking into
  account Lemma \ref{lemma:uguaglianzaP_t},
  \begin{multline*}
    \int_H \dfrac{\left|y\right|^{2m}}{1+\beta \left|y\right|^{2m}}
    \mu(dy) =\lim_{t\to \infty} P_t \phi(x) 
    =\lim_{t \to \infty} \int_H \dfrac{\left|y\right|^{2m}}{1+\beta
      \left|y\right|^{2m}} \mu_{t,x}({\rm d}y) 
    \\
    \leq \lim_{t \to \infty }
    C_m(1+e^{-m\omega_1 t}\left|x\right|^{2m}),
  \end{multline*} 
  which yields \eqref{eq:stima-integrale}.
\end{proof}

Applying formula \eqref{eq:crescita-pol-F_eta} we immediately obtain
the following

\begin{corollary}\label{cor:stima-int-F}
  We have
  \begin{equation}
    \int_H \left|F_{\eta}(x)\right|^{2m} \, \mu({\rm d}x) <\infty.
  \end{equation}
\end{corollary}
             
The corollary implies that $N_0\phi \in L^p(H,\mu)$ for all $\phi \in
\mathcal{E}_{A}(H)$ as required.  We can now show that $N_0\phi
=N\phi$ for all $\phi \in \mathcal{E}_{A}(H)$.

\begin{lemma}\label{lemma:Ito-N}
  For any $\phi \in \mathcal{E}_{A}(H)$ we have
  \begin{align}\label{eq:Ito-N}
    \EE \left[\phi(X(t,x))\right]=\phi(x)+ \EE \left[\int_0^t N_0
      \phi(X(s,x)) \, {\rm d}s\right], \quad t\geq 0, \, x\in H.
  \end{align}
  Moreover $\phi \in D(N)$ and $N_0\phi=N \phi$.
\end{lemma}

\begin{proof}
  Equality \eqref{eq:Ito-N} follows easily by applying It\^o's
  formula. It remains to prove that $\mathcal{E}_{A}(H)\subset D(N)$
  and $N_0\phi=N\phi$.  Since it holds that
  \begin{align*}
    \lim_{h \to \infty}\dfrac{1}{h}\left(P_t\phi(x)-\phi(x)\right) =
    N_0\phi(x)
  \end{align*}
  pointwise, it is enough to show that 
  \begin{align*}
    \dfrac{1}{h}(P_h \phi-\phi), \quad h\in(0,1],
  \end{align*}
  is equibounded in $L^p(H,\mu)$. 
                 
  We note that, in view of \eqref{eq:L-phi} and \eqref{eq:Ito-N}, for
  any $x\in H$ we have
  \begin{align*}
    \left|P_h \phi (x)-\phi(x)\right| \leq \int_0^h \EE
    \left[a+b\left|X(s,x)\right| +
    \left|\phi\right|_0\left|F(X(s,x))\right| \right]
    \, {\rm d}s.
  \end{align*}
  By H\"older's inequality we find that
  \begin{align*}
    |P_h\phi(x)&-\phi(x)|^p \leq h^{p-1} \int_0^h \EE
    \left[a+b\left|X(s,x)\right| + \left|\phi\right|_0
      \left|F(X(s,x))\right| \right]^p \, {\rm d}s
    \\
    &\leq c_ph^{p-1} \int_0^h \EE
    \left[a+b\left|X(s,x)\right|\right]^p \, {\rm d}s + c_p h^p
    \left|\phi\right|_0^p \int_0^h \EE\left|F(X(s,x))\right|^p \, {\rm
      d}s
    \\
    &= c_ph^{p-1} \int_0^h P_s\left(a+b\left|\cdot\right|\right)^p \,
    {\rm d}s + c_p h^{p-1} \left|\phi\right|_0^p \int_0^h
    P_s\left|F(\cdot)\right|^p \, {\rm d}s.
  \end{align*}
  Integrating with respect to $\mu$ over $H$ and taking into account
  the invariance of $\mu$, the above formula yields
  \begin{equation*}
    \left|P_h \phi -\phi\right|^p_{L^p(H,\mu)} 
    \leq h^{p} \int_H
    \left[\left(a+b\left|x\right|^p\right)+\left|\phi\right|^p
      \left|F(x)\right|^p\right]\mu({\rm d}x) < \infty, 
  \end{equation*}
  thanks to Corollary \ref{cor:stima-int-F}.  Consequently $1/h(P_h
  \phi-\phi)$ is equibounded in $L^p(H,\mu)$ as claimed.
\end{proof}
             
\begin{theorem}
  Assume that Hypothesis \ref{hp:1} holds. Then $N$ is the closure of
  $N_0$ in $L^p(H,\mu)$.
\end{theorem}
             
\begin{proof}
  By Lemma \ref{lemma:Ito-N}, $N$ extends $N_0$. Since $N$ is
  dissipative (it is the infinitesimal generator of a $C_0$
  contraction semigroup), so it is $N_0$. Consequently $N_0$ is
  closable.  Let us denote by $\bar{N}_0$ its closure. We have to show
  that $\bar{N}_0=N$.
                
  Let $\lb >0$ and $f\in \mathcal{E}_{A_{\eta}}(H)$. Consider the
  approximating equation
  \begin{equation}\label{eq:propN_0}
    \lb \phi_{\ep}-L\phi_{\ep}-\left\langle
      F_{\eta,\ep},D\phi_{\ep}\right\rangle=f, \quad \ep > 0
  \end{equation}
  By \cite[Theorem 3.21]{dp:kolmogorov} we have that Equation
  \eqref{eq:propN_0} has a unique solution $\phi_{\ep} \in C_b^1(H)$
  given by
  \begin{equation*}
    \phi_{\ep}(x)= {\mathbb E} \int_0^1 e^{-\lb t}f(X_{\ep}(t,x)) \, {\rm d}t, \quad \forall x\in H.
  \end{equation*}  
  Moreover, for all $h\in H$ we have
  \begin{align*}
    \left\langle D\phi_{\ep}(x),h\right\rangle = \int_0^{\infty}
    e^{-\lb t} \EE \left[\left\langle
        Df(X_{\ep}(t,x)),DX_{\ep}(t,x)[h]\right\rangle\right] \, {\rm d}t.
  \end{align*}
  and by \cite[Proposition 11.2.13]{dpz:second-order} we have
  \begin{align*}
    \left\|DX_{\ep}(t,x)\right\|_{\mathcal{L}(H)} \leq e^{\omega t};
  \end{align*}
  consequently we obtain
  \begin{align*}
    \left|D\phi_{\ep}(x)\right|_H \leq \dfrac{1}{\lb
      -\omega}\left\|Df\right\|_0.
  \end{align*}
  Arguing as in \cite[Theorem 11.2.14]{dpz:second-order} we can write
  \eqref{eq:propN_0} as
  \begin{align*}
    \lb \phi_{\ep}-\bar{N}_0 \phi_{\ep} =f+ \left\langle
      F_{\eta,\ep}-F,D\phi_{\ep}\right\rangle.
  \end{align*}             
  We claim that
  \begin{align*}
    \lim_{\ep \to 0} \left\langle F_{\eta,\ep}-F,
      D\phi_{\ep}\right\rangle =0 \quad \textrm{in} \ L^p(H,\mu).
  \end{align*}
  In fact, we have
  \begin{align*}
    \int_H \left|\left\langle
        F_{\eta,\ep}(x)-F(x),D\phi_{\ep}\right\rangle\right|^p \, 
    \mu({\rm d}x)\leq \dfrac{1}{\lambda-\omega}\left\|Df\right\|_0^p
    \int_H \left|F_{\eta,\ep}(x)-F(x)\right|^p \, \mu({\rm d}x).
  \end{align*}
  Clearly,
  \begin{align*}
    \lim_{\ep \to 0} \left|F_{\eta,\ep}(x)-F(x)\right|^p =0, \quad
    \mu-a.e.
  \end{align*}
  Moreover
  \begin{equation*}
    \left|F_{\eta,\ep}(x)-F(x)\right|^p \leq 2\left|F(x)\right|^p,
    \quad x\in H.
  \end{equation*}
  Therefore, the claim follows from the dominated convergence theorem,
  since
  \begin{equation*}
    \int_H \left|F(x)\right|^p \, \mu({\rm d}x) <\infty
  \end{equation*}
  in virtue of Corollary \ref{cor:stima-int-F}. In conclusion we have
  proved that the closure of the range of $\lb - \bar{N}_0$ includes
  $\mathcal{E}_A(H)$ which is dense in in $L^p(H,\mu)$. Now the
  theorem follows from Lumer-Phillips theorem.
\end{proof}

\section*{References}

\def\cprime{$'$}


\begin{thebibliography}{10}

\bibitem{dp:kolmogorov}
G.~Da~Prato.
\newblock {\em Kolmogorov equations for stochastic {PDE}s}.
\newblock Advanced Courses in Mathematics. CRM Barcelona. Birkh\"auser Verlag,
  Basel, 2004.

\bibitem{dpz:stochastic}
G.~Da~Prato and J.~Zabczyk.
\newblock {\em Stochastic equations in infinite dimensions}, volume~44 of {\em
  Encyclopedia of Mathematics and its Applications}.
\newblock Cambridge University Press, Cambridge, 1992.

\bibitem{dpz:ergodicity}
G.~Da~Prato and J.~Zabczyk.
\newblock {\em Ergodicity for infinite-dimensional systems}, volume 229 of {\em
  London Mathematical Society Lecture Note Series}.
\newblock Cambridge University Press, Cambridge, 1996.

\bibitem{dpz:second-order}
G.~Da~Prato and J.~Zabczyk.
\newblock {\em Second order partial differential equations in {H}ilbert
  spaces}, volume 293 of {\em London Mathematical Society Lecture Note Series}.
\newblock Cambridge University Press, Cambridge, 2002.

\bibitem{fitzhugh}
R.~FitzHugh.
\newblock Impulses and physiological stales in theoretical models of nerve
  membrane.
\newblock {\em Biophys. j.}, 1:445--466, 1961.

\bibitem{hodgkin/huxley}
A.L. Hodgkin and A.F. Huxley.
\newblock A quantitative description of membrane current and its application to
  conduction and excitation in nerve.
\newblock {\em J. Physiol.}, 117(2):500--544, 1952.

\bibitem{keener/sneyd}
J.~Keener and J.~Sneyd.
\newblock {\em Mathematical physiology}, volume~8 of {\em Interdisciplinary
  Applied Mathematics}.
\newblock Springer-Verlag, New York, 1998.

\bibitem{murray}
J.~D. Murray.
\newblock {\em Mathematical biology. {I}}, volume~17 of {\em Interdisciplinary
  Applied Mathematics}.
\newblock Springer-Verlag, New York, third edition, 2002.
\newblock An introduction.

\bibitem{nagel/1989}
R.~Nagel.
\newblock Towards a ``matrix theory'' for unbounded operator matrices.
\newblock {\em Math. Z.}, 201(1):57--68, 1989.

\bibitem{nagumo}
J.S. Nagumo, S.~Arimolo, and S.~Yoshizawa.
\newblock An active pulse transmission line simulating nerve axon.
\newblock {\em Proc. IRE}, 50:2061--2071, 1962.

\bibitem{ouhabaz}
El~Maati Ouhabaz.
\newblock {\em Analysis of heat equations on domains}, volume~31 of {\em London
  Mathematical Society Monographs Series}.
\newblock Princeton University Press, Princeton, NJ, 2005.

\end{thebibliography}
\end{document}